\documentclass[12pt]{amsart}
\usepackage{amssymb}
\usepackage{amsfonts}
\usepackage{amssymb,latexsym}
\usepackage{enumerate}
\usepackage{mathrsfs}
\makeatletter
\@namedef{subjclassname@2010}{%
  \textup{2010} Mathematics Subject Classification}
\makeatother

  [2002/01/19 v2.2g %
    AMS font definitions%
  ]
\DeclareFontFamily{U}{euf}{}
\DeclareFontShape{U}{euf}{m}{n}{%
  <5><6><7><8><9>gen*eufm%
  <10><10.95><12><14.4><17.28><20.74><24.88>eufm10%
  }{}
\DeclareFontShape{U}{euf}{b}{n}{%
  <5><6><7><8><9>gen*eufb%
  <10><10.95><12><14.4><17.28><20.74><24.88>eufb10%
  }{}

  [2002/01/19 v2.2g %
    AMS font definitions%
  ]
\DeclareFontFamily{U}{msb}{}
\DeclareFontShape{U}{msb}{m}{n}{%
  <5><6><7><8><9>gen*msbm%
  <10><10.95><12><14.4><17.28><20.74><24.88>msbm10%
  }{}

  [2002/01/19 v2.2g %
    AMS font definitions%
  ]
\DeclareFontFamily{U}{msa}{}
\DeclareFontShape{U}{msa}{m}{n}{%
  <5><6><7><8><9>gen*msam%
  <10><10.95><12><14.4><17.28><20.74><24.88>msam10%
  }{}

\newtheorem{theorem}{Theorem}[section]

\newtheorem{corollary}[theorem]{Corollary}

\theoremstyle{definition}

\newtheorem{remark}[theorem]{Remark}

\numberwithin{equation}{section}
\begin{document}

\title[Gauss sums over some matrix groups] {Gauss sums over some matrix groups}

\author{Yan Li and Su Hu}

\address{Department of Applied Mathematics, China Agriculture University, Beijing 100083, China}
\email{liyan\_00@mails.tsinghua.edu.cn}
\address{Department of
Mathematics, Korea Advanced Institute of Science and Technology
(KAIST), 373-1 Guseong-dong, Yuseong-gu, Daejeon 305-701, South
Korea} \email{hus04@mails.tsinghua.edu.cn, husu@kaist.ac.kr}


\begin{abstract}In this note, we give explicit expressions of Gauss sums
for general (resp. special) linear groups over finite fields, which
involves Gauss sums (resp. Kloosterman sums). The key ingredient is
averaging such sums over Borel subgroups. As  applications, we count
the number of invertible matrices of zero-trace over finite fields
and we also improve two bounds by Ferguson, Hoffman, Luca, Ostafe
and Shparlinski in [ Some additive combinatorics problems in matrix
rings, Rev. Mat. Complut. (23) 2010, 501--513 ].
\end{abstract}

\maketitle

\def\C{\mathbb C_p}
\def\BZ{\mathbb Z}
\def\Z{\mathbb Z_p}
\def\Q{\mathbb Q_p}
\def\C{\mathbb C_p}
\def\BZ{\mathbb Z}
\def\Z{\mathbb Z_p}
\def\Q{\mathbb Q_p}
\def\psum{\sideset{}{^{(p)}}\sum}
\def\pprod{\sideset{}{^{(p)}}\prod}

\section{Introduction}
The Gauss sums for classical groups over a finite field have been
extensively studied by  Kim in several articles ~\cite{Kim,
Kim1,Kim2, Kim3, Kim4,Kim5,Kim6,Kim7, Kim8, Kim9, Kim10, Kim11,
Kim12, Kim13}, and more recently, he also applied the Gauss sum for
special linear groups over finite fields to coding
theory~\cite{Kim14}.

 Let $q$ be a
power of a prime number $p$ and $k=\mathbb{F}_{q}$ be the finite
field with $q$ elements. Let
 $\lambda$
be a fixed nonprincipal additive character of $\mathbb{F}_{q}$, e.g,
take $$\lambda(x)={\rm exp}\Big(\frac{2\pi i}{p}{\rm
tr}_{\mathbb{F}_{q}/\mathbb{F}_{p}}(x)\Big),\ \ \forall\ x\in
\mathbb{F}_{q},$$ and $\chi$ be a multiplicative character of
$\mathbb{F}_{q}^{*}$.

Given two matrices $U=(u_{ij}),V=(v_{ij})\in {\rm M}_{n}(k)$, their
product is defined by
\begin{equation}\begin{aligned}\label{f0}U\cdot V
= \sum_{\substack{1\leq i\leq n \\
1\leq j\leq n}}u_{ij}v_{ij}.\end{aligned}\end{equation}

For $U$ being a nonzero matrix of ${\rm M}_{n}(k)$, let
\begin{equation}\label{f1}\begin{aligned}G_{{\rm GL}_{n}(k)}(U,\chi,\lambda)
&= \sum_{X\in{\rm GL}_{n}(k)}\chi\left(\det X\right)\lambda(U \cdot
X ),\end{aligned}\end{equation}
and\begin{equation}\label{f2}\begin{aligned}G_{{\rm
SL}_{n}(k)}(U,\lambda) &=
\sum_{X\in\textrm{SL}_{n}(\mathbb{F}_{q})}\lambda(U \cdot X ).
\end{aligned}\end{equation} These sums can be viewed as the general linear group
and special linear group analogues of classical Gauss sums. For
brevity, if $\chi=1$ is trivial, we will write $G_{{\rm
GL}_{n}(k)}(U,\lambda)$ instead of $G_{{\rm
GL}_{n}(k)}(U,1,\lambda)$.

Kim \cite{Kim} got the formulae of $G_{{\rm
GL}_{n}(k)}(I,\chi,\lambda)$ and $G_{{\rm SL}_{n}(k)}(I,\lambda)$ by
using the Bruhat decomposition of ${\rm GL}_{n}(k)$ and
$\textrm{SL}_{n}(k)$. As Kim remarked, these formulae already
appeared in the work of Eichler \cite{Eichler} and Lamprecht
\cite{Lamprecht}.(See the introduction of \cite{Kim}). Fulman
\cite{Fulman} also got the same result for $G_{{\rm
GL}_{n}(k)}(I,\chi,\lambda)$ by using the technique of generating
functions.

In this note, by using orthogonality of characters of finite abelian
groups, we present an explicit expressions for the sums (\ref{f1})
and (\ref{f2}). The main idea in our approach is averaging the sums
(\ref{f1}) and (\ref{f2}) over Borel subgroups, i.e, the group of
upper triangular matrices. (See Theorems \ref{thm1} and \ref{thm2}
below). As a consequence, we give the upper bounds of the sums
(\ref{f1}) and (\ref{f2}) (in the case $\chi=1$). (See Corollaries
\ref{cor1} and \ref{cor2} below).

By using several results from algebraic geometry, in particular,
Skorobogatov's estimates of character sums along algebraic
varieties~\cite{Skoro}, Ferguson, Hoffman, Luca, Ostafe and
Shparlinski~\cite{Sh} also  provided another upper bounds of the
sums (\ref{f1}) and (\ref{f2}) and used them to study some additive
combinatorics problems in matrix rings. Their bounds have been
applied by us to study some uniform distribution properties of some
matrix groups. (See \cite{HL}). Our bounds given in this note
improve their bounds. (See Remarks \ref{rem1} and \ref{rem2} below).

Finally, as an application, we count the number of invertible
matrices of zero-trace over finite fields.

\section{Main results}
\subsection{The case of ${\rm GL}_{n}(k)$}The equation (\ref{f1}) can
be rewritten as \begin{equation}\label{f3}\begin{aligned}G_{{\rm
GL}_{n}(k)}(U,\chi,\lambda) &= \sum_{X\in{\rm
GL}_{n}(k)}\chi\left(\det X\right)\lambda({\rm tr}\ U^tX
),\end{aligned}\end{equation} where $U^t$ is the transpose of $U$
and ``${\rm tr}$'' stands for the trace of the matrix.

Replacing $U$ by $PUQ$ in (\ref{f3}) with $P,Q\in {\rm GL}_{n}(k)$,
we get
\begin{equation}\label{f4}\begin{aligned}&\ \ \ G_{{\rm
GL}_{n}(k)}(PUQ,\chi,\lambda)
\\=&\sum_{X\in{\rm GL}_{n}(k)}\chi\left(\det X\right)\lambda({\rm
tr}\ Q^tU^tP^tX )\\=&\sum_{X\in{\rm GL}_{n}(k)}\chi\left(\det
X\right)\lambda({\rm tr}\ U^tP^tXQ^t)\\=&\bar{\chi}\left(\det
PQ\right)\sum_{X\in{\rm GL}_{n}(k)}\chi\left(\det
P^tXQ^t\right)\lambda({\rm tr}\ U^tP^tXQ^t)\\=&\bar{\chi}\left(\det
PQ\right)G_{{\rm GL}_{n}(k)}(U,\chi,\lambda)
.\end{aligned}\end{equation} Therefore,
\begin{equation}\label{f5}\begin{aligned} G_{{\rm GL}_{n}(k)}(U,\chi,\lambda)=\chi
\left(\det PQ\right)G_{{\rm GL}_{n}(k)}(PUQ,\chi,\lambda).
\end{aligned}\end{equation}

Let $u$ be the rank of $U$. There exist $P,Q\in {\rm GL}_{n}(k)$
such that
\begin{equation}\label{f6}\begin{aligned}
PUQ=\left(
      \begin{array}{cc}
        I_u & 0 \\
        0 & 0 \\
      \end{array}
    \right)
,\end{aligned}\end{equation} where $I_u$ is the $u\times u$ identity
matrix. If $u<n$, additionally, we can also require $P,Q\in {\rm
SL}_{n}(k)$.

Combining equations (\ref{f5}) and (\ref{f6}), we get
\begin{equation}\label{f7}\begin{aligned}
 \ \ \ \ \ G_{{\rm
GL}_{n}(k)}(U,\chi,\lambda)=&\left\{
\begin{array}{llll} \displaystyle\bar{\chi}(\det U)\sum_{X\in{\rm
GL}_{n}(k)}\chi\left(\det X\right)\lambda({\rm tr}_{u}X) &\ &\mathrm{if}\ u=n\\
\displaystyle\sum_{X\in{\rm GL}_{n}(k)}\chi\left(\det
X\right)\lambda({\rm tr}_{u}X)&\ &\mathrm{if}\ u<n
\end{array}\right.
,
\end{aligned}\end{equation}
where
\begin{equation}\label{f8}\begin{aligned} {\rm tr}_{u}X=\sum_{i=1}^{u}x_{ii},\ {\rm for}\ X=(x_{ij})\in {\rm M}_{n}(k).
\end{aligned}\end{equation}
So it suffices to calculating the sum
\begin{equation}\label{f9}\begin{aligned}\sum_{X\in{\rm
GL}_{n}(k)}\chi\left(\det X\right)\lambda({\rm tr}_{u}X),\ {\rm
for}\ 1\leq u\leq n.
\end{aligned}\end{equation}

Let ${\rm B}_{n}(k)$ be the Borel subgroup of ${\rm GL}_{n}(k)$,
i.e., the group of upper triangular invertible matrices. The
following is the key step of our approach.

\begin{equation}\label{f10}\begin{aligned}&\ \ \sum_{X\in{\rm
GL}_{n}(k)}\chi\left(\det X\right)\lambda({\rm
tr}_{u}X)\\=&\frac{1}{(q-1)^nq^{\left( n\atop2 \right) }}\sum_{B\in
{\rm B}_{n}(k)}\sum_{X\in{\rm GL}_{n}(k)}\chi\left(\det
BX\right)\lambda({\rm tr}_{u}BX)
\\=&\sum_{X\in{\rm
GL}_{n}(k)}\frac{\chi\left(\det X\right)}{(q-1)^nq^{\left( n\atop2
\right) }}\sum_{B\in {\rm B}_{n}(k)}\chi\left(\det
B\right)\lambda({\rm tr}_{u}BX)
\\=&\sum_{\left(x_{ij}\right)\in{\rm
GL}_{n}(k)}\frac{\chi\left(\det
\left(x_{ij}\right)\right)}{(q-1)^nq^{\left( n\atop2 \right)
}}\sum_{\left(b_{ij}\right)\in {\rm
B}_{n}(k)}\chi\left(\prod_{i=1}^{n}b_{ii}\right)\lambda(\sum_{\substack{i\leq
j\\ i\leq u}}b_{ij}x_{ji})
\\=&\sum_{\left(x_{ij}\right)\in{\rm
GL}_{n}(k)}\frac{\chi\left(\det
\left(x_{ij}\right)\right)}{(q-1)^nq^{\left( n\atop2 \right)
}}\prod_{i=1}^{u}\sum_{b_{ii}\neq
0}\chi\left(b_{ii}\right)\lambda(b_{ii}x_{ii})\cdot\prod_{i=u+1}^{n}\sum_{b_{ii}\neq
0}\chi\left(b_{ii}\right)
\\&\ \ \ \ \ \ \ \ \ \ \cdot\prod_{\substack{i<j\\ i\leq
u}}\sum_{b_{ij}\in
k}\lambda(b_{ij}x_{ji})\cdot\prod_{\substack{i<j\\
i> u}}\sum_{b_{ij}\in k}1.
\end{aligned}\end{equation}

If $\chi$ is not principal and $u<n$, then we have
\begin{equation}\label{f111}\sum_{X\in{\rm
GL}_{n}(k)}\chi\left(\det X\right)\lambda({\rm
tr}_{u}X)=0\end{equation}
 because
\begin{equation*}\sum_{b_{ii}\neq 0}\chi\left(b_{ii}\right)=0,\ {\rm for}\ u+1\leq i\leq n.
\end{equation*}

So we only need to consider the remaining two cases: $u=n$ or
$\chi=1$.

Firstly, assume $u=n$. The sum (\ref{f10}) equals to
\begin{equation}\label{f11}\begin{aligned}
\sum_{\left(x_{ij}\right)\in{\rm GL}_{n}(k)}\frac{\chi\left(\det
\left(x_{ij}\right)\right)}{(q-1)^nq^{\left( n\atop2 \right)
}}\prod_{i=1}^{n}\sum_{b_{ii}\neq
0}\chi\left(b_{ii}\right)\lambda(b_{ii}x_{ii})
\cdot\prod_{i<j}\sum_{b_{ij}\in k}\lambda(b_{ij}x_{ji}).
\end{aligned}\end{equation}
For $i<j$,
\begin{equation*}\begin{aligned}
\sum_{b_{ij}\in k}\lambda(b_{ij}x_{ji})\neq0\ {\rm if\ and\ only\
if}\ x_{ji}=0.
\end{aligned}\end{equation*}
Therefore, (\ref{f11}) equals to
\begin{equation}\label{f12}\begin{aligned}
\ \ &\sum_{\left(x_{ij}\right)\in{\rm
B}_{n}(k)}\frac{1}{(q-1)^n}\prod_{i=1}^{n}\sum_{b_{ii}\neq
0}\chi\left(b_{ii}x_{ii}\right)\lambda(b_{ii}x_{ii})
\end{aligned}\end{equation}
So we get
\begin{equation}\label{f13}\begin{aligned}
\sum_{X\in{\rm GL}_{n}(k)}\chi\left(\det X\right)\lambda({\rm
tr}_{u}X)=q^{\left( n\atop2 \right)}G(\chi,\lambda)^n,
\end{aligned}\end{equation}where
\begin{equation}\label{f14}\begin{aligned}
G(\chi,\lambda)=\sum_{x\in k^*}\chi(x)\lambda(x)
\end{aligned}\end{equation} is the classical Gauss sum for $k=\mathbb{F}_q$.

Secondly, assume $\chi$ is principal. The sum (\ref{f10}) equals to
\begin{equation}\label{f15}\begin{aligned}\sum_{\left(x_{ij}\right)
\in{\rm GL}_{n}(k)}\frac{1}{(q-1)^uq^{\left( n\atop2 \right) -\left(
n-u\atop2 \right)}}\prod_{i=1}^{u}\sum_{b_{ii}\neq
0}\lambda(b_{ii}x_{ii})\cdot\prod_{\substack{i<j\\ i\leq
u}}\sum_{b_{ij}\in k}\lambda(b_{ij}x_{ji}).
\end{aligned}\end{equation}

The terms in the summation (\ref{f15}) over
$X=\left(x_{ij}\right)\in{\rm GL}_{n}(k)$ are nonzero if and only if
\begin{equation}\label{f16}\begin{aligned}
X=\left(
      \begin{array}{cc}
        A & B \\
        0 & D \\
      \end{array}
    \right)
,\ {\rm for}\ A\in {\rm B}_{u}(k)\ {\rm and}\ D\in{\rm
GL}_{n-u}(k).\end{aligned}\end{equation} In that case, they all
equal to
\begin{equation*}\begin{aligned}\frac{(-1)^u}{(q-1)^{u}}.
\end{aligned}\end{equation*}
Therefore, \begin{equation}\label{f17}\begin{aligned} &
\sum_{X\in{\rm GL}_{n}(k)}\chi\left(\det X\right)\lambda({\rm
tr}_{u}X)\\=&\frac{(-1)^u}{(q-1)^{u}}\#\{(A,B,D)|A\in {\rm
B}_{u}(k),\ B\in {\rm M}_{u\times (n-u)}(k),\  D\in{\rm
GL}_{n-u}(k)\}\\=&\frac{(-1)^u}{(q-1)^{u}}\left((q-1)^{u}q^{\left(
n\atop2 \right)-\left( n-u\atop2
\right)}\right)\cdot\prod_{i=0}^{n-u-1}\left(q^{n-u}-q^i\right)
\\=&q^{\left( n\atop2
\right)}(-1)^u\prod_{i=1}^{n-u}\left(q^{i}-1\right).
\end{aligned}\end{equation}
Putting equations $(\ref{f7})$, $(\ref{f111})$, $(\ref{f13})$ and
$(\ref{f17})$ all together, we get
\begin{theorem}\label{thm1}Let $u$ be the rank of $U$ and $\lambda$ be nontrivial. Then\begin{equation*}\begin{aligned}
 \ \ \ \ \ G_{{\rm
GL}_{n}(k)}(U,\chi,\lambda)=&\left\{
\begin{array}{llll} \displaystyle\bar{\chi}(\det U)q^{\left( n\atop2 \right)}G(\chi,\lambda)^n &\ &\mathrm{if}\ u=n,\\
\displaystyle(-1)^uq^{\left( n\atop2
\right)}\prod_{i=1}^{n-u}\left(q^{i}-1\right)&\ &\mathrm{if}\
\chi=1,\\ \displaystyle0&\ &\mathrm{if}\ u<n,\ \chi\neq1,
\end{array}\right.
\end{aligned}\end{equation*}
where $G(\chi,\lambda)$ is the classical Gauss sum defined in
$(\ref{f14})$.
\end{theorem}
Letting $u=1$ in Theorem \ref{thm1}, we get the following corollary.
\begin{corollary}\label{cor1} Uniformly over all nonzero matrices
$U\in\textrm{M}_{n}(\mathbb{F}_{q})$ and nontrivial additive
characters $\lambda$, we have
$$G_{{\rm
GL}_{n}(k)}(U,\lambda)= O(q^{n^{2}-n}),$$ where the implied constant
in the symbol $``O"$ depends only on $n$.
\end{corollary}
\begin{remark}\label{rem1} Ferguson, Hoffman, Luca, Ostafe
and Shparlinski in \cite{Sh} obtained $G_{{\rm
GL}_{n}(k)}(U,\lambda)=O(q^{n^{2}-5/2})$. (See Lemma 3 of
\cite{Sh}). Corollary \ref{cor1} improves their bounds. Moreover,
from our proof, it is easily seen that the bound $O(q^{n^{2}-n})$
can not be improved. So Lemma 3 of \cite{Sh} should be stengthed as
$$G_{{\rm GL}_{n}(k)}(U,\lambda)=O(q^{n^{2}-5/2}),\ {\rm for}\ \
n>2.$$
\end{remark}

\subsection{The case of ${\rm SL}_{n}(k)$}

If $P,Q\in {\rm SL}_{n}(k)$, the same argument as (\ref{f4}) shows
that
\begin{equation}\label{s1}\begin{aligned} G_{{\rm SL}_{n}(k)}(U,\lambda)=
G_{{\rm SL}_{n}(k)}(PUQ,\lambda).
\end{aligned}\end{equation}

If $u<n$, from (\ref{f6}), we can assume $P,Q\in {\rm SL}_{n}(k)$,
then
\begin{equation}\label{s2}\begin{aligned} G_{{\rm SL}_{n}(k)}(U,\lambda)=
\sum_{X\in{\rm SL}_{n}(k)}\lambda({\rm tr}_{u}X).
\end{aligned}\end{equation}
Let $D_h\in {\rm GL}_{n}(k)$ be the diagonal matrix ${\rm
Diag}(1,1,\ldots,1,h)$, where $h\in k^*$. Every element $Y$ of ${\rm
GL}_{n}(k)$ can be uniquely written as $Y=D_hX$ with $X\in {\rm
SL}_{n}(k)$ and $h=\det X$. So
\begin{equation*}\label{s4}\begin{aligned}
\sum_{X\in{\rm SL}_{n}(k)}\lambda({\rm tr}_{u}X)=\sum_{X\in{\rm
SL}_{n}(k)}\frac{1}{q-1}\sum_{h\neq 0}\lambda({\rm
tr}_{u}D_{h}X)=\frac{1}{q-1}\sum_{Y\in{\rm GL}_{n}(k)}\lambda({\rm
tr}_{u}Y).
\end{aligned}\end{equation*}
Therefore, from Theorem \ref{thm1}, we get
\begin{equation}\label{s5}\begin{aligned} G_{{\rm
SL}_{n}(k)}(U,\lambda)=\frac{1}{q-1} G_{{\rm
GL}_{n}(k)}(U,1,\lambda)=\displaystyle(-1)^uq^{\left( n\atop2
\right)}\prod_{i=2}^{n-u}\left(q^{i}-1\right).
\end{aligned}\end{equation}

Now we consider the case $u=n$. Let $\widetilde{{\rm B}}_n(k)$ be
the Borel subgroup of ${\rm SL}_{n}(k)$, i.e., the group of upper
triangular matrices with determinate $1$. Using the same method as
in the case of ${\rm GL}_{n}(k)$, we get
\begin{equation}\label{s7}\begin{aligned}&\ \ \ G_{{\rm SL}_{n}(k)}(U,\lambda)\\ =&\sum_{X\in{\rm
SL}_{n}(k)}\lambda({\rm tr}\ U^tX )\\=&\sum_{\det X=\det
U}\lambda({\rm tr}\ X )\\=&\frac{1}{(q-1)^{n-1}q^{\left( n\atop2
\right) }}\sum_{B\in \widetilde{{\rm B}}_n(k)}\sum_{\det X=\det
U}\lambda({\rm tr}\ BX )
\\=&\sum_{\det X=\det
U}\frac{1}{(q-1)^{n-1}q^{\left( n\atop2 \right) }}\sum_{B\in
\widetilde{{\rm B}}_n(k)}\lambda({\rm tr}\ BX )
\\=&\sum_{\det (x_{ij})=\det
U}\frac{1}{(q-1)^{n-1}q^{\left( n\atop2 \right) }}\sum_{(b_{ij})\in
\widetilde{{\rm B}}_n(k)}\lambda(\sum_{i\leq j}b_{ij}x_{ji})
\\=&\sum_{\det (x_{ij})=\det
U}\frac{1}{(q-1)^{n-1}q^{\left( n\atop2 \right) }}
\sum_{b_{11}\cdots b_{nn}=1}\lambda(\sum_{i=1}^{n}b_{ii}x_{ii})
\cdot\prod_{i<j}\sum_{b_{ij}\in k}\lambda(b_{ij}x_{ji})
\\=&\sum_{\substack{(x_{ij})\in {\rm B}_n(k)\\ x_{11}\cdots x_{nn}=\det
U}}\frac{1}{(q-1)^{n-1}} \sum_{b_{11}\cdots
b_{nn}=1}\lambda(\sum_{i=1}^{n}b_{ii}x_{ii})
\\=& q^{\left( n\atop2
\right)}K_n(\lambda,\det U),
\end{aligned}\end{equation}where
\begin{equation}\label{s8}\begin{aligned}
K_n(\lambda,y)=\sum_{x_{1}x_{2}\cdots
x_{n}=y}\lambda(x_1+x_2+\cdots+x_n),\ {\rm for}\ y\in k^*,
\end{aligned}\end{equation} is the Kloosterman sum for $k=\mathbb{F}_q$.

Summing up, we get
\begin{theorem}\label{thm2}Let $u$ be the rank of $U$ and $\lambda$ be nontrivial. Then\begin{equation*}\begin{aligned}
 \ \ \ \ \ G_{{\rm SL}_{n}(k)}(U,\lambda)=&\left\{
\begin{array}{llll} \displaystyle q^{\left( n\atop2
\right)}K_n(\lambda,\det U) &\ &\mathrm{if}\ u=n,\\
\displaystyle(-1)^uq^{\left( n\atop2
\right)}\prod_{i=2}^{n-u}\left(q^{i}-1\right)&\ &\mathrm{if}\ u<n.
\end{array}\right.
\end{aligned}\end{equation*}
\end{theorem}
\begin{corollary}\label{cor2} Uniformly over all nonzero matrices
$U\in\textrm{M}_{n}(\mathbb{F}_{q})$ and nontrivial additive
characters $\lambda$, we have
\begin{equation*}\begin{aligned}
 \ \ \ \ \ G_{{\rm SL}_{n}(k)}(U,\lambda)=&\left\{
\begin{array}{llll}O(1) &\ &\mathrm{if}\ n=
1,\\ O(q^{\frac{3}{2}}) &\ &\mathrm{if}\ n= 2,\\O(q^{n^2-n-1}) &\
&\mathrm{if}\ n\geq 3.
\end{array}\right.
\end{aligned}\end{equation*} where the implied constant
in the symbol $``O"$ depends only on $n$.
\end{corollary}
\begin{proof}The case: $n=1$ is trivial. So we assume $n\geq 2$.
From Delinge's bound of Kloosterman sum (see Example 2
in~\cite{kowalski}), we get $K_n(\lambda,y)=O(q^{\frac{n-1}{2}})$,
where $y\in k^*$. Therefore, by Theorem \ref{thm2}, we get
$$G_{{\rm
SL}_{n}(k)}(U,\lambda)=
O(\max\{q^{\frac{n^2-1}{2}},q^{n^{2}-n-1}\}).
$$Note that $n^{2}-n-1>(n^2-1)/2$ if and only if $n\geq 3$. So we get the desired
bound.
\end{proof}
\begin{remark}\label{rem2} Ferguson, Hoffman, Luca, Ostafe
and Shparlinski in \cite{Sh} obtained $G_{{\rm
SL}_{n}(k)}(U,\lambda)=O(q^{n^{2}-2})$. (See Lemma 4 of \cite{Sh}).
Corollary \ref{cor2} improves their bounds.
\end{remark}
\section{Counting Invertible matrices with given trace}
For $\beta\in k$, let
\begin{equation}\label{c1}\begin{aligned} N_{\beta}=\#\{X\in {\rm GL}_{n}(k)|\ {\rm tr}
X=\beta\}.
\end{aligned}\end{equation}

 The usual way of computing $N_{\beta}$ involves the Bruhat
 decomposition of ${\rm GL}_{n}(k)$, e.g, see
 \cite[Prop.\,1.10.15]{Stanley}. In this note, as an application of Theorem \ref{thm1}, we
 can calculate $N_{\beta}$ purely by the method of exponential sums.

 Since trace is a linear function, we get $N_{\beta}=N_{h\beta},$
 for $h,\beta\in k^*$. So $N_{h}=N_{1},$ for $h\in k^*$. Then, we have
\begin{equation}\label{c2}\begin{aligned}N_0+(q-1)N_1=\#{\rm
Gl}_n(k)=\prod_{i=0}^{n-1}(q^n-q^i)\end{aligned}
\end{equation}
and
 \begin{equation}\label{c3}\begin{aligned}G_{{\rm
GL}_{n}(k)}(I,\lambda) =& \sum_{X\in{\rm GL}_{n}(k)}\lambda({\rm
tr}\ X )\\=&N_{0}\lambda(0)+N_{1}\sum_{h\in
k^*}\lambda(h)=N_{0}-N_{1}.\end{aligned}\end{equation} Combining
Theorem \ref{thm1}, equations (\ref{c2}) and (\ref{c3}), we get
\begin{theorem}\label{thm3} Let $h\in k^*$. Then
\begin{equation*}\begin{aligned}N_0=&q^{\left( n\atop2
\right)}(q-1)\left((-1)^n+\prod_{i=2}^{n}\left(q^{i}-1\right)\right)/q,\\N_h=&q^{\left(
n\atop2
\right)}(q-1)\left((-1)^{n-1}/(q-1)+\prod_{i=2}^{n}\left(q^{i}-1\right)\right)/q.
\end{aligned}\end{equation*}
\end{theorem}
\textbf{Acknowledgement} The first author is supported by the
National Natural Science Foundation of China (Grant No. 11001145 and
Grant No. 11071277). The second  author is supported by the National
Research Foundation of Korea(NRF) grant funded by the Korea
government(MEST) (2011-0001184).


\begin{thebibliography}{??}
\parskip=0pt
\itemsep=0pt
\bibitem{Eichler}M. Eichler, Allgemeine Kongruenz-Klasseneinteilungen der
Ideale einfacher Algebren \"{u}ber algebraischen Zahlk\"{o}rpern und
ihre L-Reihen. J. Reine Angew. Math. 179 (1937), 227--251.
\bibitem{Sh} R. Ferguson, C. Hoffman, F. Luca, A. Ostafe, I. E. Shparlinski, Some additive combinatorics problems in matrix
rings, Rev. Mat. Complut. (23) 2010, 501--513.
\bibitem{Fulman} J. Fulman,  A New Bound for Kloosterman Sums, arXiv:math/0105172v3.
\bibitem{HL} S. Hu, Y. Li, On a uniformly distributed phenomenon in matrix
groups, arXiv:1103.3928.
\bibitem{Kim1}D. S. Kim, Gauss sums for classical groups over a finite field, Number theory, geometry and related topics (Iksan City, 1995), 97--101, Pyungsan Inst. Math. Sci., Seoul, 1996.
\bibitem{Kim2} D. S. Kim, I.-S. Lee, Gauss sums for ${\rm O}^{+}(2n,q)$, Acta
Arith. 78 (1996), 75--89.
\bibitem{Kim3}D. S. Kim, Gauss sums for
${\rm O}^-(2n,q)$, Acta Arith. 80 (1997), 343--365.
\bibitem{Kim}D. S. Kim, Gauss sums for general and special
linear groups over a finite field, Arch. Math. (Basel) 69 (1997),
297-304.
\bibitem{Kim4}D. S. Kim, Y.-H. Park, Gauss sums for orthogonal groups over a finite field of characteristic two, Acta Arith. 82 (1997), 331--357.
\bibitem{Kim5}D. S. Kim, Gauss sums for ${\rm U}(2n+1,q^2)$, J.
Korean Math. Soc. 34 (1997), 871--894.
\bibitem{Kim6}D. S. Kim,
Gauss sums for ${\rm O}(2n+1,q)$, Finite Fields Appl. 4 (1998),
62--86.
\bibitem{Kim7}D. S. Kim, Gauss sums for symplectic groups over a finite field, Monatsh.
Math. 126 (1998), 55--71.
\bibitem{Kim8} D. S. Kim, Exponential sums for
symplectic groups and their applications, Acta Arith. 88 (1999),
155--171.
\bibitem{Kim9}D. S. Kim, Hodges' Kloosterman sums. Number theory and related topics (Seoul, 1998), 119--133, Yonsei Univ. Inst. Math. Sci., Seoul,
2000.
\bibitem{Kim10}D. S. Kim, Exponential sums for ${\rm O}^-(2n,q)$ and their applications, Acta Arith. 97 (2001),67--86.
\bibitem{Kim11}D. S. Kim, Exponential sums for ${\rm O}(2n+1,q)$ and their applications, Glasg. Math. J. 43 (2001),
219--235.
\bibitem{Kim12}D. S. Kim, Exponential sums for ${\rm O}^+(2n,q)$ and their applications, Acta Math. Hungar. 91 (2001), 79--97.
\bibitem{Kim13}D. S. Kim, Sums for $U(2n,q^2)$ and their applications, Acta Arith. 101
(2002), 339--363.
\bibitem{Kim14}D. S. Kim, Codes associated with special linear groups and power moments of multi-dimensional Kloosterman sums, Ann. Mat. Pura Appl. (4) 190
(2011), 61--76.
\bibitem{kowalski} E. Kowalski, Some  aspects  and  applications  of the Riemann
hypothesis over finite fields, Milan J. of Mathematics, 78 (2010),
179--220.
\bibitem{Kondo}T. Kondo, On Gaussian sums attached to the general linear
groups over finite fields. J. Math. Soc. Japan 15 (1963), 244--255.
\bibitem{Lamprecht}E. Lamprecht, Struktur und Relationen allgemeiner Gau{\ss}cher
Summen in endlichen Ringen I, II. J. Reine Angew. Math. 197 (1957),
1--48.
\bibitem{Skoro} A. N. Skorobogatov, Exponential sums, the geometry of hyperplane sections, and some Diophantine problems, Isr. J. Math. 80 (1992),
359--379.
\bibitem{Stanley} R. P. Stanley, Enumerative Combinatorics, vol. I, second
edition, available at http://math.mit.edu/~rstan/ec/ec1.
\end{thebibliography}
\end{document}